\documentclass[11pt]{article}
\usepackage[a4paper,margin=0.88in]{geometry}
\usepackage{amsmath,amssymb,amsthm,mathtools}

\usepackage{enumitem}
\usepackage{microtype}
\usepackage[hidelinks]{hyperref}
\theoremstyle{plain}
\newtheorem{theorem}{Theorem}[section]
\newtheorem{lemma}[theorem]{Lemma}

\newtheorem{proposition}[theorem]{Proposition}

\theoremstyle{definition}
\newtheorem{definition}[theorem]{Definition}
\theoremstyle{remark}

\theoremstyle{plain}
\newtheorem{question}[theorem]{Question}
\theoremstyle{plain}

\usepackage{xcolor}
\usepackage[nameinlink,capitalise]{cleveref}
\definecolor{darkblue}{RGB}{0, 0, 100}
\definecolor{darkorange}{RGB}{200, 100, 0}
\definecolor{repairorange}{RGB}{220, 110, 0}

\newcommand{\Cvec}{\vec C}

\newcommand{\eps}{\varepsilon}
\newcommand{\ind}{\alpha}
\newcommand{\semid}{\delta^0}

\newcommand{\N}{\mathbb N}

\newcommand{\cP}{\mathcal P}
\newcommand{\bfe}{\mathbf e}
\newcommand{\bfi}{\mathbf i}
\newcommand{\calI}{\mathcal I}

\title{Ramsey--Tur\'an Factors of Arbitrarily Oriented Cycles\\ in Oriented Graphs}
\author{ Jia Zhou$^1$,  Yunshu Gao$^1$\thanks{Corresponding author: gysh2004@gmail.com}\\{\small
$^1$ School of Mathematics and Statistics, Ningxia University, Yinchuan 750021, China }}

\begin{document}
\maketitle

\begin{abstract}
Let $\Cvec$ be any orientation of the cycle $C_\ell$ which is not
directed.  We prove that, for every {integer $\ell\ge3$ and every $\mu>0$}, there is a real $\gamma$ such that every sufficiently
large oriented graph $D$ with $\ell\mid |D|$, $\semid(D)\geq(1/4+\mu)|D|$ and $\ind(D)\leq \gamma |D|$
has a $\Cvec$-factor. {The constant $1/4$ is asymptotically tight.} This proof establishes Ramsey–Turán‑type lattice‑absorption lemmas and an almost‑covering theorem via the oriented‑tree embedding lemma under chromatic‑number constraints.
\end{abstract}

\vspace{1ex}


\section{Introduction}

Factor existence constitutes a fundamental theme in extremal graph theory, bridging degree thresholds, divisibility obstructions, and graph tiling structures. Let \(H\) be a graph (resp., digraph) and let \(G\) be a  graph (resp., digraph). An \textbf{\(H\)-factor} is a collection of
vertex-disjoint copies of \(H\),
which covers all the vertices of \(G\). Hell and Kirkpatrick~\cite{HellKirkpatrick} proved that the decision problem for the existence of an $H$-factor is NP-complete whenever $H$ contains a connected component of order at least three. This result implies that a complete structural characterization of $H$-factorable graphs is generally unattainable, which motivates the extensive study of sufficient degree conditions for $H$-factors. Dirac's theorem and
{the} Hajnal-Szemer\'edi theorem are the classical results in this direction. Specifically, Dirac~\cite{Dirac} proved that every
\(n\)-vertex graph \(G\) with
$\delta(G)\geq \frac{n}{2}$
contains a \textbf{Hamilton cycle}, which is a cycle containing every vertex of $G$. The Hajnal--Szemer\'edi theorem~\cite{HajnalSzemeredi} states that, whenever \(r\mid n\),
every \(n\)-vertex graph \(G\) satisfying $\delta(G)\geq \left(1-\frac1r\right)n$ contains a \(K_r\)-factor. The case \(r=3\) was proved earlier by Corr\'adi and Hajnal~\cite{CorradiHajnal}.

Ramsey--Tur\'an theory provides a natural framework for problems concerning $H$-factors. A set $U\subseteq V(G)$ is \textbf{independent} if $G[U]$ contains no {edges}, and the independence number $\alpha(G)$ denotes the maximum cardinality of such a set. Ramsey--Tur\'an theory was initiated by Erd\H{o}s and S\'os~\cite{ErdosSos}. They established the asymptotic Ramsey--Tur\'an density for odd cliques and posed the corresponding problem for even cliques under the constraint that $\alpha(G)$ is bounded. This problem was settled later by Erd\H{o}s, Hajnal, S\'os and Szemer\'edi. This
type of problem is referred to as a Ramsey-Turán problem, and has been studied extensively, see \cite{CHY2025,ErdosHajnalSosSzemeredi,HMWY2024,HNY2026,SimonovitsSos}. In 2016, Balogh, Molla, and Sharifzadeh~\cite{BaloghMollaSharifzadeh} posed the following Ramsey--Tur\'an type variant of the Hajnal--Szemer\'edi theorem, and resolved the case $r = 3$ for sufficiently large graphs.
\begin{question}\label{q:rt-kr-tiling}~\cite{BaloghMollaSharifzadeh}
Let $G$ be an $n$-vertex graph with {$\alpha(G) = o(n)$}. What is the minimum degree condition on $G$ that guarantees a $K_r$-factor in $G$ for $r \ge 3$?
\end{question}

In 2021, Knierim and Su~\cite{KnierimSu} resolved Question \ref{q:rt-kr-tiling} for all $r\geq 4$, and Chen, Han, Wang, and Yang~\cite{CHWY2024} further generalized it to the theory of general $H$-factors.
 
These results have also motivated
an extensive study of $H$-factors in digraphs,  see for instance~\cite{BaloghLoMolla,ChenLuWangZhang2026,CKM2014, DeBiasioLoMollaTreglown,KeevashSudakov, Treglown,Wang2000, Yuster2003}. Exploring degree conditions that guarantee cycle factors with arbitrary orientations is a natural generalization of classical cycle‑factor problems. Czygrinow, Kierstead and Molla \cite[see Corollary 1.5.7]{molla} proved that sufficiently large digraphs with minimum total degree is at least $(3n-3)/2$ can be partitioned into $k$-vertex tiles covering all orientations of $k$-cycles. For odd oriented cycles, Molla and Treglown \cite{MollaTreglown2026} established the asymptotically tight semidegree threshold for $\vec{C}$-factors, where $\vec{C}$ is an orientation of a cycle. Here, the \textbf{minimum semidegree} $\delta^0(D)$ of a digraph $D$ is the smallest value obtained by taking, for each vertex, the smaller of its out-degree and in-degree.  
\begin{theorem}\label{thm:mollatreglown}\cite{MollaTreglown2026}
Given $k\in\mathbb{N}$ and $\eta>0$, there exists $n_0 = n_0(k,\eta)\in\mathbb{N}$ such that the following holds for all $n\ge n_0$ and $2k+1\mid n$. Let $\Cvec$ be an orientation of a cycle on $2k+1$ vertices. If $D$ is an $n$-vertex digraph with
\[
\delta^0(D)\ge \frac{(k+1)n}{2k+1}+\eta n,
\]
then $D$ contains a $\Cvec$-factor.
\end{theorem}

An \textbf{oriented graph} is a digraph obtained by assigning a direction to each edge of a  graph. The thresholds in the above results can often be lowered when restricted to oriented graphs. In 2025, Lo proved the following.
\begin{theorem}\cite{lo}\label{lo}
For all $\varepsilon>0$, there exists $\ell_0=\ell_0(\varepsilon)$ such that for all $\ell\ge \ell_0$ and any oriented cycle $\Cvec $ on $\ell$ vertices, any oriented graph $D$ on $k\ell$ vertices with $\delta^0(D)\ge (3/8+\varepsilon)k\ell$ contains a $\Cvec $-factor. Moreover, the constant $3/8$ cannot be improved.
\end{theorem}

An \textbf{anti-directed cycle} \(C_{2s}^{\mathrm{ad}}\) is an orientation of an even cycle in which the
directions of consecutive edges alternate, which is a  typical example of a non‑directed cycle orientation. Grant~\cite{Grant1980} initiated the study of anti-directed cycles in digraphs by considering degree conditions forcing an
anti-directed Hamilton cycle. DeBiasio and Molla~\cite{DeBiasioMolla}
subsequently determined the corresponding minimum semidegree threshold in
digraphs.

More recently, Chen~\cite{ChenAntiDirected} established a minimum {semidegree} condition guaranteeing the existence of a {$C_{2s}^{\mathrm{ad}}$-factor} in oriented graphs. 
For every digraph $D$, $\boldsymbol{\alpha(D)}$ denotes the maximum size of a set of vertices with no arcs between any two distinct vertices within the set.
 Recently, analogous
Ramsey--Tur\'an type questions have also been considered for $H$-factor in digraphs. Chen, Kou, Ma, and Zhang~\cite{ChenKouMaZhang2026} proved a  Ramsey--Tur\'an result: every sufficiently large oriented graph $D$ with $3\mid |D|$,   $\delta^0(D)\geq(\tfrac14+o(1))|D|$, and independence number $\alpha(D)=o(|D|)$ admits a $TT_3$-factor, and their bound is asymptotically tight. Additionally, Wang, Wang and Yan determined an  asymptotically tight minimum total degree for the existence of $TT_3$-factors in digraphs~\cite{WWY2026}, and put the following question.
\begin{question}\label{question:Wang}\cite{WWY2026}
Let $D$ be an $n$-vertex digraph with $\alpha(D)=o(n)$ and let $H$ be a fixed digraph such that $|H|\mid n$. What is the minimum {semidegree} $\delta^0(D)$ condition on $D$ that guarantees a $H$-factor in $D$?
\end{question}
 An orientation of a cycle is \textbf{directed} whenever all its arcs follow the same cyclic orientation. A directed cycle contains  a source nor a sink, while any other orientation of a cycle admits at least one source and one sink. Removing either a source or a sink from such an oriented cycle yields an oriented path. This provides the primary motivation for our study of non-directed cycle orientations.  Motivated by these results, we prove the 
following results.

\begin{theorem}\label{thm:main}
Let $\Cvec$ be an orientation of $C_\ell$, $\ell\geq3$, which is not a
directed cycle.  For every $\mu>0$ there exist $\gamma>0$ and
$n_0\in\N$ such that the following holds.  If $n\geq n_0$, $\ell\mid n$, and
$D$ is an oriented graph on $n$ vertices satisfying
\[
\semid(D)\geq\left(\frac14+\mu\right)n
  \qquad \text{and} \qquad \ind(D)\leq\gamma n,
\]
then $D$ has a $\Cvec$-factor.
\end{theorem}

Theorem~\ref{thm:main} contains the Ramsey--Tur\'an counterpart of Theorem~\ref{lo}, reducing the asymptotic semidegree threshold from $3n/8$ to $n/4$ under bounded independence number, excluding directed cycles. {For an odd cycle $\ell=2k+1$ with $k\ge2$, Theorem~\ref{thm:main} reduces the threshold $\frac{k+1}{2k+1}$ established in Theorem~\ref{thm:mollatreglown} to $1/4$ under the small independence assumption. For $\ell=3$, Theorem~\ref{thm:main} recovers the $TT_3$-factor result of Chen, Kou, Ma, and Zhang.} Proposition \ref{prop:lower} illustrates that  the bound of minimum semidegree in Theorem~\ref{thm:main} is asymptotically best possible. Thus Theorem~\ref{thm:main} {resolves the minimum semidegree version of} Question~\ref{question:Wang} {for all oriented graph $D$ and all oriented cycles $H=\Cvec$ that are not directed}.

 The proof of Theorem \ref{thm:main} utilizes the directed regularity lemma, reachability, closed partitions, robust index vectors, lattice merging, absorbing sets, and the almost covering lemma. Using the oriented tree embedding lemma under chromatic‑number constraints, we build the targeted robust index vectors. A careful construction yields elementary transferrals, and by leveraging the known lattice merging lemma we establish the Ramsey–Turán type lattice absorption lemmas and the almost covering theorem.

\paragraph{Organisation.} Section 2 introduces the notation used throughout the paper.   Section 3   presents the required diregularity and lattice‑based absorption tools.
We then establish the Ramsey--Tur\'an type absorbing lemma, and derive a Ramsey--Tur\'an type almost covering lemma (Lemma~\ref{lem:almost-cover}) to complete the proof. Section 4 provides the lower bound construction.

\section{Notation}

Notation not specified in this section is consistent with that in \cite{book}. Denote $ \boldsymbol{\mathbb N}$ for the set of positive integers. For integers $a\leq b$, let $ \boldsymbol{[a,b]}:=\{i\in\mathbb Z:a\leq i\leq b\}$, $ \boldsymbol{[b]}:=[1,b].$ If $h\in\mathbb N$, then $ \boldsymbol{h\mathbb N}$ is the set of positive
multiples of $h$. We use the notation $ \boldsymbol{0<a\ll b\ll c}$
to mean that $a$ is chosen sufficiently small in terms of $b$, and then
$b$ sufficiently small in terms of $c$.  For a set $W$ and
an integer $k\geq0$, $ \boldsymbol{\binom{W}{k}}$ denotes the family of all $k$-subsets
of $W$.    For a family $\mathcal{F}$ of sets, $ \boldsymbol{V(\mathcal{F})}$ denotes the union of the vertex sets of all members of $\mathcal{F}$.
  Given sets $A$ and $B$, the symbol $ \boldsymbol{A\mathbin{\dot\cup} B}$ stands for their disjoint union, i.e., $A\cap B=\emptyset$.

All digraphs are
finite and loopless, and no multiple arcs are allowed.      For a digraph $D$, we
write $V(D)$ and $A(D)$ for its vertex set and arc set, respectively, and
write $ \boldsymbol{|D|}:=|V(D)|$ for its \textbf{order}.   If
$U\subseteq V(D)$, then $ \boldsymbol{D[U]}$ denotes the subdigraph of $D$ induced by
$U$, and $ \boldsymbol{D-U}:=D[V(D)\setminus U]$. And if $F$ is a subdigraph of $D$, we
also write $ \boldsymbol{D-F}:=D-V(F)$.
 For $v\in V(D)$, let
\[
  \boldsymbol{N_D^+(v)}:=\{x\in V(D):vx\in A(D)\},
 \qquad
  \boldsymbol{N_D^-(v)}:=\{x\in V(D):xv\in A(D)\}
\]
be the \textbf{out-neighbourhood} and \textbf{in-neighbourhood} of $v$, respectively, and
let $  \boldsymbol{d_D^+(v)}:=|N_D^+(v)|$, $
  \boldsymbol{d_D^-(v)}:=|N_D^-(v)|.$
For $X\subseteq V(D)$, put
\[
  \boldsymbol{N_D^+(v,X)}:=N_D^+(v)\cap X,
 \qquad
 \boldsymbol{ N_D^-(v,X)}:=N_D^-(v)\cap X,
\]
and use $ \boldsymbol{d_D^+(v,X)}$ and $ \boldsymbol{d_D^-(v,X)}$ for their cardinalities.  The subscript is omitted when the digraph is clear.  The \textbf{minimum
out-degree}, \textbf{minimum in-degree}, \textbf{minimum semidegree} and \textbf{minimum total degree}
of $D$ are, respectively,
\[
  \boldsymbol{\delta^+(D)}:=\min_{v\in V(D)}d_D^+(v),
 \qquad
  \boldsymbol{\delta^-(D)}:=\min_{v\in V(D)}d_D^-(v),
\]
\[
  \boldsymbol{\delta^0(D)}:=\min\{\delta^+(D),\delta^-(D)\},
 \qquad
  \boldsymbol{\delta(D)}:=\min_{v\in V(D)}\bigl(d_D^+(v)+d_D^-(v)\bigr).
\]


 We write $ \boldsymbol{TT_r}$ for the \textbf{transitive tournament} on $r$ vertices; thus
$TT_r$ has an ordering $v_1,\ldots,v_r$ in which $v_iv_j$ is an arc for
every $i<j$.   


\section{Proof of Theorem \ref{thm:main}}

In this section, we first present several preliminary tools and lemmas required for the proof, before giving the complete proof.

\subsection{Diregularity}

To obtain desired Ramsey--Tur\'an absorbing lemma (Lemma \ref{cor:absorber}), we introduce the following powerful tool, that is, the Diregularity lemma, which is a version of the Regularity lemma for digraphs due to Alon and Shapira \cite{AlonShapira}. Its proof is quite similar to the undirected version. The \textbf{density} of a bipartite graph $G=(A,B)$ with vertex classes $A$ and $B$ is defined to be
\[
\boldsymbol{d_G(A,B)}:=\frac{e_G(A,B)}{|A||B|}.
\]
We often write $d(A,B)$ if this is unambiguous. Given $\varepsilon>0$, we say that $G$ is \textbf{$\boldsymbol{\varepsilon}$-regular} if for all subsets $X\subseteq A$ and $Y\subseteq B$ with $|X|>\varepsilon|A|$ and $|Y|>\varepsilon|B|$ we have that $|d(X,Y)-d(A,B)|<\varepsilon$. Similarly, a directed pair $(X,Y)$ is \textbf{$\boldsymbol{\varepsilon}$-regular} if
for all $X'\subseteq X$ and $Y'\subseteq Y$ with
$|X'|\geq\varepsilon|X|$ and $|Y'|\geq\varepsilon|Y|$, we have $|d_D(X',Y')-d_D(X,Y)|<\varepsilon.$ It is \textbf{$\boldsymbol{(\varepsilon,d)}$-regular} if it is $\varepsilon$-regular and
$d_D(X,Y)\geq d$. The following Slicing Lemma is required.

\begin{lemma}[Koml\'os and Simonovits~\cite{KomlosSimonovits}, Slicing Lemma]\label{lem:slicing}
Assume $(V_1,V_2)$ is $\varepsilon$-regular with density $\beta$. For some $\alpha\ge \varepsilon$, let $V_1'\subseteq V_1$ with $|V_1'|\ge \alpha|V_1|$ and $V_2'\subseteq V_2$ with $|V_2'|\ge \alpha|V_2|$. Then $(V_1',V_2')$ is $\varepsilon'$-regular with $\varepsilon':=\max\{2\varepsilon,\varepsilon/\alpha\}$ and for its density $\beta'$ we have $|\beta'-\beta|<\varepsilon$.
\end{lemma}

Given partitions $V_0,V_1,\dots,V_k$ and $U_1,\dots,U_\ell$ of the vertex set of some graph, we say that \textbf{$\boldsymbol{V_0,V_1,\dots,V_k}$ refines $\boldsymbol{U_1,\dots,U_\ell}$} if for all $V_i$ with $i\ge 1$ there exists some $U_j$ for some $j\in[\ell]$ that contains $V_i$. Note that $V_0$ need not be contained in any $U_j$, so this is weaker than the usual notion of refinement of partitions.

We shall use the following standard variant of the degree form of the Diregularity Lemma, in which both a lower bound on the number of clusters and a bounded initial partition are prescribed; it follows from the usual degree form by the standard refinement argument.

\begin{lemma}[Young \cite{Young2005}, Degree form of the Diregularity lemma]\label{regular-lemma}
For every $\varepsilon > 0$ and positive integers $M_0,q$ there are numbers $M$ and $n_0$ such that if
\begin{itemize}
    \item $G$ is a digraph on $n \geq n_0$ vertices,
    \item $\{U_1,\dots,U_{\ell}\}$ is a partition of the vertices of $G$ with $\ell \leq q$,  and
    \item $\beta \in [0,1]$ is any real number,
\end{itemize}
then there is a partition of the vertices of $G$ into $V_0,V_1,\dots,V_k$ and a spanning subdigraph $G'$ of $G$ such that the following hold:
\begin{itemize}
    \item  $M_0 \leq k \leq M$;
    \item $|V_0| \leq \varepsilon n$;
    \item $|V_1| = \cdots = |V_k| = m:=\frac{n-|V_0|}{k}$;
    \item  $V_0,V_1,\dots,V_k$ refines the partition $U_1,\dots,U_{\ell}$;
    \item $d^+_{G'}(x) > d^+_G(x) - (\beta + \varepsilon)n$ for all vertices $x \in V(G)$;
    \item $d^-_{G'}(x) > d^-_G(x) - (\beta + \varepsilon)n$ for all vertices $x \in V(G)$;
    \item for all $i = 1,\dots,k$ the digraph $G'[V_i]$ is empty;
    \item for all integers $1\le i,j\le k$ with $i\neq j$, the bipartite graph with vertex classes $V_i$ and $V_j$ and edge set $E_{G'}(V_i,V_j)$ is $\varepsilon$-regular and has density either $0$ or at least $\beta$.
\end{itemize}
\end{lemma}

The vertex sets $V_1,\dots,V_k$ are called \textbf{clusters}, $V_0$ is called the \textbf{exceptional set} and the vertices in $V_0$ are called \textbf{exceptional vertices.} The last condition of the lemma says that all pairs of clusters are $\varepsilon$-regular in both directions (but possibly with different densities). We call the spanning digraph $G' \subseteq G$ given by the Diregularity lemma \textbf{the pure digraph}.

\begin{lemma}[Kelly et al.\ \cite{KellyKuhnOsthus}]\label{reduce-lemma}
For every $\varepsilon \in (0,1)$, there exist numbers $N:=N(\varepsilon)$ and $n_0=n_0(\varepsilon)$ such that the following holds.
Let $\beta \in [0,1]$ with $\varepsilon \leq \beta/2$, let $G$ be an oriented graph of order $n \geq n_0$ and let $R'$ be the reduced digraph with parameters $(\varepsilon,\beta)$ obtained by applying the Diregularity lemma to $G$ with $N$ as the lower bound on the number of clusters.
Then $R'$ has a spanning oriented subgraph $R$ such that  $\delta^+(R) \geq \bigl(\delta^+(G)/|G| - \beta - 3\varepsilon\bigr)|R|$ and  $\delta^-(R) \geq \bigl(\delta^-(G)/|G| - \beta - 3\varepsilon\bigr)|R|$.
\end{lemma}

The oriented graph $R$ given by Lemma \ref{reduce-lemma} is called the \textbf{reduced oriented graph} with parameters $(\varepsilon,\beta)$. Notably, the reduced graph $R$ inheriting these degree bounds is guaranteed to be oriented only if the host graph is an oriented graph. This is the primary motivation for restricting our host graph setting to oriented graphs in this paper.

\subsection{Lattice‑based absorbing method}
Another useful tool is the lattice‑based absorbing method~\cite{HMWY2024}, for which we first require the notion of $H$-reachability introduced in~\cite{HMWY2024} and originating from~\cite{LoMarkstrom}.

\begin{definition}[Reachability and closedness]
Let $m,t\in\N$,  $G$ and $H$ be two (resp., oriented) graphs, and let $u,v\in V(G)$ be distinct.  We say that $u$ and
$v$ are \textbf{$\boldsymbol{(H,m,t)}$-reachable} if, for every set
\[
 W\subseteq V(G)\setminus\{u,v\}
 \quad\text{with}\quad |W|\leq m,
\]
there exists a set
\[
 Q\subseteq V(G)\setminus(W\cup\{u,v\}),
 \qquad |Q|\leq ht-1,
\]
such that both $G[Q\cup\{u\}]$ and $G[Q\cup\{v\}]$ contain
$ H$-factors.  Such a set $Q$ is called a \textbf{$  \boldsymbol{H}$-connector} for
$u$ and $v$. A set $U\subseteq V(G)$ is \textbf{$\boldsymbol{( H,m,t)}$-closed} if every two distinct
vertices of $U$ are $(H,m,t)$-reachable in $G$.  If two vertices $u,v\in V(G)$ are $(H,m,1)$‑reachable, then we say \textbf{$\boldsymbol{u}$ is $1$-reachable to $\boldsymbol{v}$}.
\end{definition}

Reachability is monotone, that is, if $m'\leq m$ and $t'\geq t$, then every
$(H,m,t)$-reachable pair is also
$(H,m',t')$-reachable.

We also use the following notion introduced by  Keevash and Mycroft~\cite{KeevashMycroft}. Let $G$ be an $n$-vertex (resp., oriented) graph with a vertex partition $\mathcal{P}=\{V_1,\dots,V_p\}$ of $V(G)$, for some integer $p\ge 1$.  For a set
$X\subseteq V(G)$, define
\[
 \boldsymbol{\bfi_{\cP}(X)}
 :=\bigl(|X\cap V_1|,\ldots,|X\cap V_p|\bigr).
\]
For $i\in[p]$, let $\boldsymbol{\bfe_i}$ denote the $i$-th unit vector, that is, for each $1\le k\le p$, the $k$-th coordinate of $\mathbf{e}_i$ satisfies
\[
(\mathbf{e}_i)_k=
\begin{cases}
1 & k=i,\\
0 & k\neq i.
\end{cases}
\]
The vector $\mathbf{e}_i$ thus represents the pattern of taking exactly one vertex from the partition block $V_i$. Hence, the vector $\mathbf{t}=s\mathbf{e}_i+s\mathbf{e}_j$ represents the condition $|X\cap V_i|=s$, $|X\cap V_j|=s$, and $|X\cap V_l|=0$ for all $l\in [p]\setminus \{i,j\}$.  A \textbf{transferral} is a vector of the form $\mathbf{e}_i-\mathbf{e}_j$ for some distinct $i,j\in[p]$. A vector $\mathbf{v}\in\mathbb{Z}^p$ is an \textbf{$\boldsymbol{s}$-vector} if all its coordinates are non-negative and their sum is $s$.

\begin{definition}[Robust index vector]
 Given $\mu>0$ and an $h$-vertex (resp., oriented) graph $H$, an $h$-vector $\mathbf{v}$ is \textbf{$\boldsymbol{(H,\mu)}$‑robust} (briefly, $\textbf{$\boldsymbol{\mu}$‑robust} $) with respect to $\cP$ if, for every
$W\subseteq V(G)$ with $|W|\leq\mu n$, the (resp., oriented) graph $G-W$
contains a copy $F$ of $H$ satisfying $\bfi_{\cP}(V(F))=\mathbf v.$
Let $\boldsymbol{\calI_{ H}^{\mu}(\cP)}$ (briefly, $\boldsymbol{\calI^{\mu}(\cP)}$) be the set of all $(H,\mu)$‑robust $h$-vectors.
\end{definition}

In undirected graphs, the following closed partition lemma (Lemma~\ref{partition-graph}), transferral merging lemma (Lemma~\ref{lem:4.9}), absorber‑production lemma (Lemma~\ref{absorber-graph}), and absorbing‑set lemma (Lemma \ref{absorb-graph}) constitute fundamental and key ingredients of the absorption method.

\begin{lemma}
[Han et al. {\cite[Lemma 4.1]{HMWY2024}}]\label{partition-graph} For any positive constants $\gamma_1,\beta_1$, $h\in \mathbb{N}$ with $h\ge 3$ and an $h$-vertex graph $H$, there exist $\beta_2=\beta_2(\gamma_1,\beta_1,h)>0$ and $t_2\in \mathbb{N}$ such that the following holds for sufficiently large $n$. Let $G$ be an $n$-vertex graph such that every vertex in $V(G)$ is $(H,\beta_1 n,1)$-reachable to at least $\gamma_1 n$ other vertices. Then there is a partition $\mathcal{P}=\{V_1,\dots,V_p\}$ of $V(G)$ with $p\le \lceil \tfrac{1}{\gamma_1}\rceil$ such that for each $i\in [p]$, $V_i$ is $(H,\beta_2 n,t_2)$-closed and $|V_i|\ge \tfrac{\gamma_1}{2}n$.
\end{lemma}

Lemma \ref{lem:4.9} allows us to iteratively merge two distinct parts into a closed one, given the existence of a
transferral.
\begin{lemma}
[Han et al. {\cite[Lemma 4.4]{HMWY2024}}]\label{lem:4.9} Given $h,t\in \mathbb{N}$ with $h\ge 3$, an $h$-vertex graph $H$ and $\beta>0$, the following holds for sufficiently large $n$. Let $G$ be an $n$-vertex graph with a partition $\mathcal{P}=\{V_1,\dots,V_p\}$ such that each $V_i$ is $(H,\beta n,t)$-closed. For distinct $i,j\in [p]$, if there exist two $(H,\beta)$‑robust $h$-vectors $\mathbf{s},\mathbf{t}$ such that $\mathbf{s}-\mathbf{t}=\mathbf{e}_i-\mathbf{e}_j$, then $V_i\cup V_j$ is $\bigl(H,\tfrac{\beta n}{2},2ht\bigr)$-closed.
\end{lemma}

The notation for absorbers and absorbing sets given in \cite{NenadovPehova2020} is adopted below.

\begin{definition}
Let $G$ be an $n$-vertex (resp., oriented) graph and let $H$ be an
$h$-vertex (resp., oriented) graph. Then
\begin{enumerate}
    \item a subset $A\subseteq V(G)$ is a \textbf{$\boldsymbol{\xi}$-absorbing set}
    for $H$ if, for every subset $U\subseteq V(G)\setminus A$ with
    $|U|\leq\xi n$ and $|A\cup U|\in h\mathbb N$, the (resp., oriented) graph
    $G[A\cup U]$ contains an $ H$-factor.

    \item For $S\in\binom{V(G)}h$ and $t\in\mathbb N$, a set
    $A_S\subseteq V(G)\setminus S$ is a \textbf{$\boldsymbol{(H, t)}$-absorber for $S$} if $|A_S|\leq h^2t$
    and both $G[A_S]$ and $G[A_S\cup S]$ contain $ H$-factors.
\end{enumerate}
\end{definition}

From the two lemmas below, one may conclude that every closed graph contains a small absorbing set.

\begin{lemma}[Han et al. {\cite[Lemma~3.10]{HMWY2024}}]\label{absorber-graph}
Given $\beta>0$, $t,h\in\mathbb{N}$ with $h\ge 3$ and an $h$-vertex graph $H$, the following holds for sufficiently large $n\in\mathbb{N}$. Let $G$ be an $n$-vertex graph such that $V(G)$ is $(H,\beta n,t)$‑closed. Then every $S\in \binom{V(G)}{h}$ has a family of at least $\dfrac{\beta}{h^2 t} n$ vertex‑disjoint $(H,t)$‑absorbers.
\end{lemma}
\begin{lemma}[Nenadov and Pehova {\cite[Lemma~2.2]{NenadovPehova2020}}] \label{absorb-graph}
Let $H$ be a graph with $h$ vertices, $\gamma>0$ and $t\in\mathbb{N}$ be constants. Then there exists $\xi:=\xi(h,t,\gamma)$ such that the following holds for sufficiently large $n$. Suppose that $G$ is a graph with $n$ vertices such that every $S\in \binom{V(G)}{h}$ has a family of at least $\gamma n$ vertex‑disjoint $(H,t)$‑absorbers. Then $G$ contains a $\xi$-absorbing set of size at most $\gamma n$.
\end{lemma}

In what follows, we use only the oriented versions of the closed-partition, transferral-merging, absorber-production, and absorbing-set statements.  These are available in the oriented-graph framework of~\cite{ChenLuWangZhang2026}; the undirected statements above are recorded only to indicate the underlying absorption framework.

\begin{lemma}[Closed partition lemma, see {\cite[Lemma 6.4]{ChenLuWangZhang2026}}]\label{lem:4.8} 
Let $\alpha,\gamma>0$ and let $\vec H$ be a fixed digraph on
$h\geq3$ vertices.  There exist $\beta=\beta(\alpha,\gamma,h)>0$ and $T\in\N$ such that the
following holds for all sufficiently large $n$. Suppose that $D$ is an $n$-vertex oriented graph and every vertex is
$(\vec H,\alpha n,1)$-reachable to at least $\gamma n$ other
vertices.  Then $V(D)$ has a partition
$\cP=\{V_1,\ldots,V_p\}$
such that $p\leq\left\lceil\frac1\gamma\right\rceil,$ $|V_i|\geq\frac\gamma2n,$
and every $V_i$ is $(\vec H,\beta n,T)$-closed, $i\in [p]$.
\end{lemma}

\begin{lemma}[Transferral merging lemma; cf. {\cite[Lemma 6.2]{ChenLuWangZhang2026}}]
\label{transferral-merging}
For every fixed oriented graph $\vec H$ on $h\geq3$ vertices, every $t\in\N$ and every $\rho>0$, there exist constants
$\beta(\rho)>0, n_0>0$ and $\gamma(\rho,t)\in\N$ with
$\gamma(\rho,t)\geq t$ such that the following holds. Let $D$ be a digraph of order $n\geq n_0$ with a partition
$\cP=\{V_1,\ldots,V_p\}$ such that every $V_i$ is
$(\vec H,\rho n,t)$-closed. If, for distinct $i,j\in[p]$, there are two $(\vec H,\rho)$-robust
$h$-vectors $\mathbf s,\mathbf t$ satisfying 
$\mathbf s-\mathbf t=\bfe_i-\bfe_j$, then
$V_i\cup V_j$ is $(\vec H,\beta(\rho)n,\gamma(\rho,t))$-closed.
\end{lemma}

\begin{lemma}[Absorber‑production lemma; cf. {\cite[Lemma 5.3]{ChenLuWangZhang2026}}
]\label{closed}
Let $\rho>0$, $t\in\mathbb N$, and let $\vec H$ be an oriented graph on
$h\geq3$ vertices. Let $D$ be a sufficiently large $n$-vertex digraph with a partition $\mathcal P=\{V_1,\ldots,V_p\}$ such that every
$V_i$ is $(\vec H,\rho n,t)$-closed. If
$S\in\binom{V(D)}h$ and $\mathbf i_{\mathcal P}(S)$ is
$(\vec H,\rho)$-robust, then $S$ has at least $\frac{\rho}{h^3t}n$
pairwise vertex-disjoint $(\vec H, t)$-absorbers. 
\end{lemma}

\begin{lemma}[Absorbing-set lemma; see {\cite[Lemma 5.2]{ChenLuWangZhang2026}}]\label{absorbing} Let $L,h\in\mathbb{N}$ with $h\ge 3$. For every $\eta>0$, there exist
$\xi=\xi(h,L,\eta)>0$ and $n_0$ such that the following holds. Suppose that $H$ and $D$ are two digraphs of orders $h$ and $n\geq n_0$, respectively. If every set $S\in \binom{V(D)}{h}$ has at least $\eta n$ disjoint $(\vec {H},L)$‑absorbers, then $D$ contains an $(\vec {H},\xi)$‑absorbing set of size at most $\eta n$.
\end{lemma}

Every orientation $\Cvec$ of a cycle which is not directed has a source and a sink.  Thus, after choosing a source $s$ and a sink $t$, the two digraphs $\Cvec-s$ and $\Cvec-t$ are oriented paths on $\ell-1$ vertices.  These are the two orientations to be embedded within the common in‑ or out‑neighbourhoods below. This is the primary motivation for our assumption that \(\Cvec\) is not directed.

\begin{lemma}[Embedding Lemma]\label{lem:tree}\cite{AddarioBerry2013}
Every $(k^2/2 - k/2 + 1)$-chromatic digraph contains every oriented tree of order $k$.
\end{lemma}

\subsection{Auxiliary lemmas developed for the proof}
In this section, we establish the required absorption lemma and almost‑covering lemma. To prove the absorption lemma, we need the following lemma, which states that every positive linear semidegree guarantees that each vertex is
$(\Cvec,\beta_1 n,1)$-reachable to linearly many other vertices.
\begin{lemma}[Reachability lemma]\label{lem:4.6}
{Let $\Cvec$ be an orientation of $C_\ell$, $\ell\ge3$, which is not directed. For every $c\in(0,1/2]$ there exist positive constants $\beta_1,\gamma_1,\gamma , n_0$ such that every $n$-vertex oriented graph $D$ of order $n\geq n_0$ with $\delta^0(D)\ge cn$ and $\alpha(D)\leq \gamma n$ has the property that every vertex in $D$ is $(\Cvec,\beta_1n,1)$-reachable to at least $\gamma_1n$ other vertices.}
\end{lemma}

\begin{proof}
Choose a source $s$ of $\Cvec$ and let $\vec P:=\Cvec-s$. Put $\beta_1:=\frac{c^2}{8}$, $ \gamma:=\frac{c^2}{8((\ell-2)^2+1)}$ and $\gamma_1:=\frac{c^2}{4}$. {For each $v\in V(D)$, double counting gives}
\begin{equation}
\sum_{u\in V(D)}|N_D^+(v)\cap N_D^+(u)|
=\sum_{x\in N_D^+(v)}d_D^-(x)\ge c^2n^2.
\end{equation}
Let $R_v:=\{u\in V(D):|N_D^+(v)\cap N_D^+(u)|\ge c^2n/2\}$. Then $|R_v|\ge c^2n/2$. Indeed,
\[
c^2n^2\le |R_v|n+(n-|R_v|)\frac{c^2n}{2}.
\]
For $u\in R_v\setminus\{v\}$ and $W\subseteq V(D)\setminus\{u,v\}$ with $|W|\le\beta_1n$, set $S:=(N_D^+(u)\cap N_D^+(v))\setminus W$. Then $|S|\ge \frac{3c^2}{8}n.$ Since $\ind(D[S])\le\ind(D)\le\gamma n$, we have
\[
 \chi(UG(D[S]))\ge\frac{|S|}{\ind(D[S])}
\ge 3((\ell-2)^2+1)>(\ell-2)^2+1.
\]
By Lemma~\ref{lem:tree}, $D[S]$ contains $\vec P$. Since every vertex of $S$ is an out-neighbour of both $u$ and $v$, adding $u$ or $v$ to this copy of $\vec P$ gives a copy of $\Cvec$. Thus $v$ is $(\Cvec,\beta_1n,1)$-reachable to every $u\in R_v\setminus\{v\}$, and hence, for sufficiently large $n$, to at least $\gamma_1n$ other vertices.
\end{proof}

We also need the following Robust Transferral Lemma to prove the absorption lemma.

\begin{lemma}[Robust Transferral Lemma]\label{lem:profiles}
For any $d>0$ and $M\in\N$, and every
$0<\varepsilon\leq\min\{d/8,1/8\}$, there are constants
$\beta,\gamma>0$ and $n_0\in\N$ such that the following holds.  Let $D$ be an oriented graph of order $n\ge n_0$ with a partition $\mathcal P=(V_1,\ldots,V_k)$.  Suppose that, for some distinct $i,j\in[k]$, there are sets $X\subseteq V_i$ and $Y\subseteq V_j$ with $|X|=|Y|=m\ge n/M$ such that $(X,Y)$ is an $(\varepsilon,d)$-regular directed pair.  If $\alpha(D)\le\gamma n$, then there exist two $(\Cvec,\beta)$-robust $\ell$-vectors
\[
\mathbf p_{ij}:=\mathbf e_i+(\ell-1)\mathbf e_j,\qquad
\mathbf q_{ij}:=(\ell-1)\mathbf e_i+\mathbf e_j,
\]
with respect to $\mathcal P$.
\end{lemma}

\begin{proof}
Choose a source $s$ and a sink $t$ of $\Cvec$, and put
$K:=(\ell-2)^2+1$.  Choose $\beta$ and $\gamma$, in the order, such that $0<\beta\ll\min\{\varepsilon/(2M),d/(8M)\} $ and $ 0<\gamma\ll d/(4MK).$  Fix an arbitrary set $W\subseteq V(D)$ with
$|W|\le\beta n$.

Because $(X,Y)$ is an $(\varepsilon , d)$-regular directed pair, it follows that all but at most $\varepsilon m$ vertices $x\in X$ satisfy
$d_D^+(x,Y)\ge(d-\varepsilon)m$.  Since $\beta n\le\varepsilon m/2$, there is a vertex
$x\in X\setminus W$ with this property.  Hence
\[
|N_D^+(x)\cap(Y\setminus W)|
\ge(d-\varepsilon)m-\beta n
\ge \frac d2m
\ge \frac d{2M}n.
\]
Let $S:=D[N_D^+(x)\cap(Y\setminus W)]$.  Since
$\alpha(S)\le\alpha(D)\le\gamma n$, we have
\[
\chi(UG(S))\ge \frac{|S|}{\alpha(S)}
\ge \frac{dn/(2M)}{\gamma n}
\ge 2K>K.
\]
By the embedding lemma, $S$ contains a copy of the oriented path $\Cvec-s$; adding $x$ gives a copy of $\Cvec$ with index vector $\mathbf p_{ij}$.

Similarly, all but at most $\varepsilon m$ vertices $y\in Y$ satisfy
$d_D^-(y,X)\ge(d-\varepsilon)m$.  Choose such a vertex $y\in Y\setminus W$ and set
$S':=N_D^-(y)\cap(X\setminus W)$.  Then
$|S'|\ge d m/2\ge dn/(2M)$, and the same chromatic estimate gives
$\chi(UG(D[S']))>K$.  Thus $D[S']$ contains a copy of $\Cvec-t$; adding $y$ gives a copy of $\Cvec$ with index vector $\mathbf q_{ij}$. Since $W$ was arbitrary, both index vectors are $(\Cvec,\beta)$-robust.
\end{proof}

\begin{lemma}[Absorbing lemma]\label{cor:absorber}
For every $\mu,\eta>0$ there are $\gamma,\xi, n_0>0$ such that every $D$ is an oriented graph of order $n\geq n_0$ satisfying $\semid(D)\geq\left(\frac14+\mu\right)n$ and  $ \ind(D)\leq\gamma n$ contains a set
$A\subseteq V(D)$, $|A|\leq\eta n$, for which $D[A\cup U]$ has a
$\Cvec$-factor for every set
 $U\subseteq V(D)\setminus A$ with $|U|\leq\xi n$ and
 $\ell\mid |A\cup U|.$
\end{lemma}

\begin{proof}
To clearly illustrate our choice of parameters, we specify the parameters in the first paragraph. Given parameters $\mu, \eta>0$, we choose the constants
\[
\mu_0,\,\beta_{3.16},\,\gamma_{3.16},\,\beta_{3.11},\,t_{3.11},\,d,\,\varepsilon ,
\]
in sequential order satisfying the following. Let $\mu_0:=\min\left\{\mu,\frac14\right\}.$ The constants  $\beta_{3.16},\gamma_{3.16}>0$ are obtained by applying Lemma~\ref{lem:4.6} with parameter $\mu_0$. The constants $\beta_{3.11}>0$ and $t_{3.11}\in\mathbb{N}$ are obtained by applying Lemma~\ref{lem:4.8}  with $\alpha:=\beta_{3.16}$, $\gamma:=\gamma_{3.16}$, and $\vec H=\Cvec$.
Choose $d$ such that $0<d\ll \min \{\mu,\beta_{3.11},\gamma_{3.16},1/\ell\}$.  Next choose
$\varepsilon$ such that $0<\varepsilon\ll \min \{d,\mu,\gamma_{3.16} \} $, so that
\begin{equation}\label{4.1}
\varepsilon\leq\frac d8,
\qquad
\varepsilon\leq\frac18,
\qquad
d+3\varepsilon\leq\frac\mu2,
\qquad
\varepsilon<\frac{\gamma_{3.16}}{4}.
\end{equation}
To decide the constant $\gamma$ in the lemma, we first use  Diregularity lemma (Lemma \ref{regular-lemma}) with $M_0:=p_0$, $q:=p_0$, $\varepsilon:=\varepsilon$ and $\beta:=d$ to obtain an upper bound $M_{\rm reg}$ for the number of parts. Then let $M_{\rm prof}:=\max\left\{2M_{\rm reg},\left\lceil\frac{2}{\gamma_{3.16}}\right\rceil\right\}.$ And then apply Lemma~\ref{lem:profiles} with $M:=M_{\rm prof}$ to get the constants $\beta_{3.17},\gamma_{3.17}>0$. Put
\[
\rho_0:=\min\left\{\beta_{3.11},\frac{\beta_{3.17}}{2\ell}\right\} \text{ and } t_0:=t_{3.11}.
\]
For $0\le r\le p_0-2$, we recursively construct the parameters $\rho_{r+1}$ and $t_{r+1}$ as follows. First, we invoke Lemma~\ref{transferral-merging} with $\rho:=\rho_r$, $t:=t_r$, and $\vec H:=\Cvec$ to yield constants $\beta_{3.12}(\rho_r)$ and $\gamma_{3.12}(\rho_r,t_r)$. We then set
\[
\rho_{r+1}:=\min\left\{\frac{\rho_r}{2},\beta_{3.12}(\rho_r)\right\},\qquad
t_{r+1}:=\gamma_{3.12}(\rho_r,t_r).
\]
Finally set
\[
\beta_*:=\rho_{p_0-1}, \qquad t_*:=t_{p_0-1}
\qquad \text{and} \qquad
\gamma:=\min\left\{\frac{\gamma_{3.17}}2,\frac1{M_{\rm prof}}\right\}.
\]
Increase $n_0$ if necessary so that all selected constants are valid for $n\ge n_0$.

Let $D$ be an oriented graph satisfying the assumptions of the lemma with $\alpha(D)\leq\gamma n$.  Since $\delta^0(D)\geq\left(\frac14+\mu\right)n\geq\mu_0n,$
Lemma~\ref{lem:4.6} implies that every vertex is
$(\Cvec,\beta_{3.16}n,1)$‑reachable to at least $\gamma_{3.16}n$ other vertices.
Hence Lemma~\ref{lem:4.8} with $\alpha:=\beta_{3.16},\gamma:=\gamma_{3.16}$ and $\vec H:=\Cvec$ gives a partition
\begin{equation}\label{4.6}
\mathcal P_0=\{V_1^0,\ldots,V_{p_0'}^0\},
\qquad p_0'\leq p_0:=\left\lceil\frac1{\gamma_{3.16}}\right\rceil,
\end{equation}
such that
\begin{equation}\label{4.7}
|V_i^0|\geq\frac{\gamma_{3.16}}{2}n
\quad\text{and}\quad
V_i^0\text{ is }(\Cvec,\beta_{3.11}n,t_{3.11})\text{-closed}
\end{equation}
for every $i\in[p_0']$.

Apply Diregularity lemma (Lemma \ref{regular-lemma}) with $M_0:=p_0$, $\{U_1,\dots,U_{p_0'}\}:=\{V_1^0,\ldots,V_{p_0'}^0\} $, $\varepsilon:=\varepsilon$, and $\beta:=d$, refining $\mathcal P_0$.  By the preceding parameter choice, the number of clusters is at most the preselected constant $M_{\rm reg}$.  This gives the resulting partition
$\mathcal P^*=\{V_0,X_1,\ldots,X_k\}$, where $|V_0|\leq\varepsilon n$ and $|X_1|=\cdots=|X_k|=m:=\frac{n-|V_0|}{k}$.  Moreover, $ k\leq M_{\rm reg}$. Let $R$ be the reduced oriented graph
supplied by Lemma~\ref{reduce-lemma} with $\varepsilon :=\varepsilon $ and $\beta:=d$.  From $ \delta^0(D)
\geq (\frac14+\mu)n$,  we obtain the estimate
\begin{equation}\label{4.8}
\begin{split}
\delta^0(R)
&\geq
\left(\frac{\delta^0(D)}{n}-d-3\varepsilon\right)|R|\geq
\left(\frac14+\mu -d-3\varepsilon\right)|R|
\stackrel{\eqref{4.1}}{\geq}
\left(\frac14+\frac\mu2\right)|R|.
\end{split}
\end{equation}
Moreover, $k=|R|$ and therefore
\begin{equation}\label{4.9}
\begin{split}
m
&=\frac{n-|V_0|}{k}
\geq\frac{(1-\varepsilon)n}{k}
\geq\frac{n}{2k}.
\end{split}
\end{equation}
The second inequality uses $\varepsilon<1/2$. Since $k\leq M_{\rm reg}$, we also have $m\geq n/(2M_{\rm reg})$.
 We next merge the closed classes. 
Suppose that, after $r$ mergers, where
$0\leq r\leq p_0'-1$, we have a partition $\mathcal P_r=\{U_1^r,\ldots,U_{p_r}^r\}$ satisfying the following three properties:
\begin{enumerate}[label=\textup{(P\arabic*)}$_r$]
\item $p_r=p_0'-r$;
\item every $U_i^r$ is a union of some parts of $\mathcal P_0$, $i\in [p_r]$;
\item every $U_i^r$ is $(\Cvec,\rho_rn,t_r)$‑closed, $i\in [p_r]$.
\end{enumerate}
These properties hold for $r=0$ by \eqref{4.6}--\eqref{4.7} and the definition of $\rho_0$ and the monotonicity of closedness. Obviously, every $U_i^r$ contains at least one original part
$V_j^0$ in $\mathcal P_0 $, and hence, by \eqref{4.7} and $|V_0|\leq\varepsilon n$, for each $i\in [p_r]$, $|U_i^r\setminus V_0|
\geq\left(\frac{\gamma_{3.16}}{2}-\varepsilon\right)n
>0.$

Assume that $p_r\geq 2$.  Since $\mathcal P^*$ refines $\mathcal P_0$, every non-exceptional cluster of $\mathcal P^*$ lies entirely within some part of $\mathcal P_r$. For $i\in[p_r]$, let
\[
Q_i^r:=\bigl\{v\in V(R)\,\big|\, \text{the cluster $X_v$ corresponding to }v \text{ in } \mathcal P^* \text{ such that }X_v\subseteq U_i^r\bigr\}.
\]
Thus, each $Q_i^r$
is nonempty, and $V(R)=Q_1^r\mathbin{\dot\cup}\cdots\mathbin{\dot\cup}Q_{p_r}^r.$

 We claim that
 \begin{equation}\label{arc}
   \begin{aligned}
   &\text{there is an $i\in [p_r]$ such that $R$ contains an arc  between $Q_i^r$ and $ Q_{j'}^r$ and }\\
   &\text{\ \ \ \ \ \ \ an arc between $Q_i^r$ and $ Q_{j''}^r$ for some $j',j''\in [p_r]$ with $j'\neq j''$,}
   \end{aligned}
 \end{equation}
(note that it is possible that $j'=i$ or $j''=i$).
 Otherwise, fix an index $i\in [p_r]$ and a vertex $v\in Q_i^r\subseteq V(R)$, all in‑ and out‑neighbours in $R$
of $v$ would lie in a same part $Q_{j'}^r$. Since $R$ is oriented,
\begin{equation}\label{4.12}
\begin{split}
|Q_{j'}^r|-1
&\geq d_R^+(v)+d_R^-(v)\geq 2\delta^0(R)
\stackrel{\eqref{4.8}}{\geq}
\left(\frac12+\mu\right)|R|
>\frac{|R|}{2}.
\end{split}
\end{equation}
Since $p_r\ge 2$, fix an arbitrary index $i'\neq i$ and a vertex $v\in Q_{i'}^r\subseteq V(R)$. Analogously, there exists a part $Q_{j''}^r$ satisfying $|Q_{j''}^r|>|R|/2$. Since inequality \eqref{arc} fails, we have $j'\neq j''$. This yields $|Q_{j'}^r|+|Q_{j''}^r|>|R|$, which is impossible given $p_r\ge 2$. This completes the proof of the claim.

Choose two such arcs. Without loss of generality, suppose that $uv',uv''\in A(R)$ with $u\in Q_i^r$, $v'\in Q_j^r$, $v''\in Q_l^r$ and $j\neq l$ (the remaining cases follow similarly). By the construction of $R$, both associated directed pairs $(V_i,V_{j'})$ and $(V_i,V_{j''})$ are $(\varepsilon,d)$‑regular, where $V_i$ denotes the cluster in $\mathcal P^*$ corresponding to $u$, $V_{j'}$ the cluster in $\mathcal P^*$ corresponding to $v'$, and $V_{j''}$ the cluster in $\mathcal P^*$ corresponding to $v''$. By \eqref{4.9},
$|V_i|=|V_{j'}|=|V_{j''}|=m\geq n/2k$.
 Since $\alpha(D)\leq\gamma n\leq\gamma_{3.17}n$, Lemma~\ref{lem:profiles} with $d:=d,M:=M_{prof},\varepsilon:=\varepsilon$
provides two $(\Cvec, \beta_{3.17})$-robust $\ell$-vectors $\mathbf p_{ij'}:=(\ell-1)\mathbf e_i+\mathbf e_{j'}$ and $
\mathbf q_{ij''}:=(\ell-1)\mathbf e_i+\mathbf e_{j''},$. Therefore,
\begin{equation}\label{4.14}
\mathbf p_{ij'}-\mathbf q_{ij''}=\mathbf e_{j'}-\mathbf e_{j''}.
\end{equation}

For every $r$, we have $\rho_r\leq\rho_0\leq\beta_{3.17}/(2\ell)$, and hence $\beta_{3.17}\geq \ell\rho_r$. Apply Lemma~\ref{transferral-merging} with $\vec H:=\Cvec$, $h:=\ell$, $t:=t_r$ and $\rho:=\rho_r$. It follows from \eqref{4.14} that
$U_{j'}^r\cup U_{j''}^r$ is $\left(\Cvec,\beta_{3.12}(\rho_r)n,\gamma_{3.12}(\rho_r,t_r)\right)$-closed. By the recursive choice
$\rho_{r+1}\leq\beta_{3.12}(\rho_r)$ and $t_{r+1}=\gamma_{3.12}(\rho_r,t_r)$, monotonicity gives that
$U_{j'}^r\cup U_{j''}^r$ is $\left(\Cvec,\rho_{r+1}n,t_{r+1}\right)$-closed. Every part $U_l^r$ with $l\in [p_r]\setminus \{j',j''\}$ is also
$(\Cvec,\rho_{r+1}n,t_{r+1})$‑closed, because
$\rho_{r+1}\leq\rho_r$ and $t_{r+1}\geq t_r$.  Replacing
$U_{j'}^r,U_{j''}^r$ by their union  $U_{j'}^r\cup U_{j''}^r$ therefore produces a new partition
$\mathcal P_{r+1}$ satisfying \textup{(P1)}$_{r+1}$--\textup{(P3)}$_{r+1}$.
At each step the number of parts decreases by one.  Hence after at most
$p_0'-1\leq p_0-1$ mergers we obtain a single closed class, namely
$V(D)$.  If the process terminates after $r_*\leq p_0-1$ steps, then
$V(D)$ is $(\Cvec,\rho_{r_*}n,t_{r_*})$-closed.
Since $\rho_{p_0-1}\leq\rho_{r_*}$ and $t_{p_0-1}\geq t_{r_*}$, the constants
$\beta_*:=\rho_{p_0-1}$ and $t_*:=t_{p_0-1}$ selected above give, by monotonicity, $V(D)$ is $(\Cvec,\beta_*n,t_*)$-closed. Further, by monotonicity,
\[
\text{$V(D)$ is $(\Cvec,\beta_*n/2,t_*)$‑closed.}
\]

 Let
$\mathcal P:=\{V(D)\}$ be the one‑part partition. Indeed, given
 $W\subseteq V(D)$ with $|W|\leq\beta_*n/2$, choose distinct
$u,v\in V(D)\setminus W$. A $\Cvec$‑connector for $u,v$ avoiding $W$
provides a copy of $\Cvec$ in
$D-W$.  Thus the unique $\ell$-vector $\ell \boldsymbol e_1$ with respect to $\mathcal P$ is $(\Cvec,\beta_*/2)$-robust. Set
$ \eta_{\rm abs}:=\min\left\{\eta,\frac{\beta_*}{4\ell ^2t_*}\right\}.
$ Applying Lemma~\ref{closed} with $\rho=\beta_*/2$ and $t=t_*$, we obtain that for each $S\in\binom{V(D)}{\ell}$, the number of vertex‑disjoint \((\vec{\mathcal C},t_*)\)‑absorbers each of size at most \(\ell ^2t_*\) is at least \(\frac{\beta_*}{4\ell ^2t_*}n \geq\eta_{\rm abs}n.\) Applying Lemma~\ref{absorbing} with
$\vec H:=\Cvec$, $L:=\ell ^2t_*$, and $\eta:=\eta_{\rm abs}$ yields a real $\xi_{\rm abs}:=\xi(\ell ,\ell ^2t_*,\eta_{\rm abs})$ and a
$\xi_{\rm abs}$‑absorbing set $A$ for $\Cvec$ satisfying $|A|\leq\eta_{\rm abs}n.$  Thus the conclusion holds with
$\gamma=\gamma$, $\xi=\xi_{\rm abs}$ and $n_0=n_0$.
\end{proof}

We also need a new almost-covering lemma, which is stated as Lemma \ref{lem:almost-cover}. To prove that, we need the following lemma.

\begin{lemma}\label{lem:two-cluster}
For every $d,\zeta>0$ there are $\varepsilon,\theta>0$ and $m_0\in\N$ such that the following holds.  Let $D$ be an oriented graph and let $X,Y$ be disjoint sets of order $m\ge m_0$.  Suppose that $(X,Y)$ is an $(\varepsilon,d)$-regular directed pair and
$\alpha(D[X\cup Y])\le\theta m$.  Then there is a set $W\subseteq X\cup Y$ with
$|W|\le6\zeta m$ such that $D[X\cup Y]-W$ has a $\Cvec$-factor.
\end{lemma}

\begin{proof}
Choose a source $s$ and a sink $t$ of $\Cvec$, and put $K:=(\ell-2)^2+1$.  Choose $\varepsilon,\theta>0$ and $m_0\in\N$ such that $0<\varepsilon\le\min\{d/8,\zeta/8\}$, $ 0<\theta\le\frac{d\zeta}{2K}$ and $\ell\le\zeta m_0/2$.
We greedily construct disjoint pairs $\mathcal{C}$  of copies $F_s,F_t$ of $\Cvec$, where $F_s$ has index vector $(1,\ell-1)$ and $F_t$ has index vector $(\ell-1,1)$ with respect to $(X,Y)$.  Thus each pair removes exactly $\ell$ vertices from each part. The lemma holds whenever $|X\setminus V(\mathcal{C})|\leq 3\zeta m$ (since our construction implies that $|X\setminus V(\mathcal{C})|=|Y\setminus V(\mathcal{C})| $). Suppose otherwise that $|X\setminus V(\mathcal{C})|> 3\zeta m$.

Set $X_0\subseteq X\setminus V(\mathcal{C}) $ and $Y_0\subseteq Y\setminus V(\mathcal{C})$ be two subsets with $|X_0|=|Y_0|=L\geq 2\zeta m+\ell$. Since $|X_0|=|Y_0|=L\ge 2\zeta m+\ell>\eps m$ and $(X,Y)$ is $(\eps,d)$-regular, we have $d_D(X_0,Y_0)\ge d-\eps$. Hence there is a vertex $x\in X_0$ with
\[
 |N_D^+(x)\cap Y_0|\geq(d-2\eps)L\geq \frac d2L\ge d\zeta m,
\]
as $\varepsilon\le d/8$. The chromatic number of the underlying graph of $D[N_D^+(x)\cap Y_0]$ is at least $ |N_D^+(x)\cap Y_0|/\alpha (D[X\cup Y])\geq K$, so it contains $\Cvec-s$.  Adding $x$ produces a copy $F_s$ of $\Cvec$ with $\ell$-vector
$(1,\ell-1)$ on $(X_0,Y_0)$.

 After deleting $F_s$, write the remaining sets as $X_1,Y_1$.  Both still
have size at least $L-\ell\ge2\zeta m$.  Similarly, there exists a vertex $y\in Y_1$ with $|N_D^-(y)\cap X_1|\geq(d-\eps)|X_1|\geq d\zeta m/2.$  The chromatic number of underlying graph of the digraph $D[N_D^-(y)\cap X_1]$ is at least $K$ (reduce $\theta$ by a
factor two if necessary), hence $D[N_D^-(y)\cap X_1]$ contains $\Cvec-t$.  Adding $y$ gives a
copy $F_t$ with $\ell$-vector $(\ell-1,1)$.  Together $F_s,F_t$ remove exactly
$\ell$ vertices from each part $X$ and $Y$, and then update $\mathcal{C}:=\mathcal{C}\cup \{F_s,F_t\}$.  Repeat until
$|X\setminus V(\mathcal{C})|\leq 3\zeta m$ vertices
remain in each part, and fewer than $6\zeta m$ in total.
\end{proof}

\begin{lemma}[Almost Covering Lemma]\label{lem:almost-cover}
For every $\mu,\xi>0$ and every integer $\ell\ge3$ there are $\gamma>0,n_0\in\N$ such that every oriented graph $D$ of order $n\ge n_0$ with
$\delta^0(D)\ge(1/4+\mu)n$ and $\alpha(D)\le\gamma n$ has a $\Cvec$-factor of $D-W$ for some $W\subseteq V(D)$ with $|W|\le\xi n$.
\end{lemma}

\begin{proof}
Fix $\mu,\xi,\ell$, and choose $\zeta$ and $ d$ so that $0<\zeta\ll\xi,\mu,1/\ell $ and $0<d\ll\mu,\zeta,1/\ell$.
Take the constants $\varepsilon_{\rm emb},\theta_{\rm emb},m_0$ from Lemma~\ref{lem:two-cluster} with parameters $d,\zeta$.  Choose $M_0$ so that $1/M_0\ll\xi$, and then choose
\[
0<\varepsilon\ll\varepsilon_{\rm emb},d,\mu,\zeta,\xi,1/M_0.
\]
Apply Lemma~\ref{regular-lemma} with $q=1$, obtaining an upper bound $M$ on the number of clusters and a partition
$(V_0,V_1,\ldots,V_k)$, where
$|V_0|\le\varepsilon n$ and $|V_1|=\cdots=|V_k|=m=(n-|V_0|)/k$, and let $R$ be the reduced oriented graph given by Lemma~\ref{reduce-lemma}.  We may choose the regularity parameters so that $\delta^0(R)\ge(1/4+\mu/2)k.$
Thus the underlying graph $G:=UG(R)$ satisfies $\delta(G)\ge2\delta^0(R)>(1/2)k.$
Dirac's theorem yields a Hamilton cycle in $G$, and hence a matching $\mathcal M$ covering all vertices of $R$ except at most one.  For every edge $uv\in E(G)$ of this matching, one of the two corresponding directed pairs is an $(\varepsilon_{\rm emb},d)$-regular pair by the definition of the reduced oriented graph.

  Since $k\le M$ and $|V_0|\le\varepsilon n< n/2$, there is
\[
m=\frac{n-|V_0|}{k}\ge\frac{n}{2M}.
\]
Choose \(\gamma>0\) satisfying \(2M\gamma\le\theta_{\rm emb}\), and take \(n_0\) sufficiently large so that \(m\ge m_0\) whenever \(n\ge n_0\). For every matched pair $(V_i,V_j)$ we have $\alpha(D[V_i\cup V_j])\le\alpha(D)\le\gamma n
\le2M\gamma m\le\theta_{\rm emb}m.$
Hence Lemma~\ref{lem:two-cluster} applies independently to every matched pair, leaving a set $W_{ij}$ of at most $6\zeta m$ vertices and covering the remainder by copies of $\Cvec$.

The exceptional set $V_0$ and the at most one unmatched part contribute at most $|V_0|+m$, while the matched pairs contribute at most
$6\zeta m\cdot |\mathcal M|\le3\zeta km\le3\zeta n$ leftover vertices.  Therefore
\[
|V_0|+m+\sum_{uv\in\mathcal M}|W_{uv}|
\le\varepsilon n+\frac{n}{M_0}+3\zeta n
<\xi n,
\]
for the chosen hierarchy.  This proves the lemma.
\end{proof}

\subsection{Proof of the theorem}

\begin{proof}[Proof of Theorem~\ref{thm:main}]
Fix $\mu>0$.  Choose $\eta>0$ so small that $0<\eta\le\min\left\{\frac\mu4,\frac12\right\}.$
Apply Corollary~\ref{cor:absorber} with parameters $\mu,\eta$ and denote the resulting constants by $\gamma_{\rm abs},\xi_{\rm abs},n^{3.18}_{0}$.  Next apply Lemma~\ref{lem:almost-cover} with parameters $\mu/2,\xi_{\rm abs},\ell$ and denote the resulting constants by $\gamma_{\rm cov},n^{3.20}_{0}$.  Set
\[
{\gamma:=\min\left\{\gamma_{\rm abs},\frac{\gamma_{\rm cov}}4\right\}}.
\]
Choose $n_0$ sufficiently large so that $n_0:=\max \{n^{3.18}_{0},n^{3.20}_{0}/(1-\eta), 1/\eta \}$.

Let $D$ satisfy the hypotheses of Theorem~\ref{thm:main}.  Lemma~\ref{cor:absorber} gives an absorbing set $A$ with $|A|\le\eta n$.  Put $D':=D-A$ and $n':=|D'|$.  Since $|A|\le\eta n$ and $\eta\le\mu/4$, we obtain that
\[
\begin{aligned}
\delta^0(D')\ge \delta^0(D)-|A|\ge (1/4+\mu)n-\eta n\ge (1/4+\mu/2)n,
\end{aligned}
\]
and, because $n'\le n$ and $n'\ge(1-\eta)n\ge n/2$, there is $(1/4+\mu/2)n\ge (1/4+\mu/2)n'.$
Moreover,
\[
\alpha(D')\le\alpha(D)\le\gamma n
\le\frac{\gamma_{\rm cov}}4n
\le\frac{\gamma_{\rm cov}}2n',
\]
where the last inequality uses $n'\ge n/2$.

Applying Lemma~\ref{lem:almost-cover} to $D'$ with parameters $\mu/2,\xi_{\rm abs},\ell$ yields a set $W\subseteq V(D')$ with
$|W|\le\xi_{\rm abs}n'$ such that $D'-W$ has a $\Cvec$-factor $\mathcal F$.  Hence $|V(\mathcal F)|$ is divisible by $\ell$. Together with $\ell\mid n$, we have $|A\cup W|=n-|V(\mathcal F)|\equiv0\pmod\ell.$
The defining property of $A$ therefore gives a $\Cvec$-factor of $D[A\cup W]$.  Together with $\mathcal F$, these copies form a $\Cvec$-factor of $D$.
\end{proof}

\section{Sharpness}
The Proposition \ref{prop:lower} implies that the bound of {semidegree} in Theorem~\ref{thm:main} is
asymptotically best possible.
\begin{proposition}\label{prop:lower}
For infinitely many multiples $n$ of $\ell$, there is an oriented graph
$D$ with $\ind(D)=2$ and $\semid(D)\geq n/4-1$ which has no
$\Cvec$-factor.
\end{proposition}

\begin{proof}
Take $n=4\ell m$ for arbitrarily large integers $m$.  Then $n/2-1$ and
$n/2+1$ are odd and neither is divisible by $\ell$.  Let $D_1,D_2$ be regular
tournaments of orders $n/2-1$ and
$n/2+1$, respectively. Let $D$ be their disjoint union.    Hence $\ind(D)=2$, and
\[
 \semid(D)=\min\left\{\frac{|D_1|-1}{2},\frac{|D_2|-1}{2}\right\}
 \geq n/4-1.
\]
The underlying cycle $C_\ell$ is connected, so every $\Cvec$-copy lies
wholly in one component.  A factor would force both component orders to be
divisible by $\ell$, a contradiction.
\end{proof}

\end{document}